\documentclass[graybox]{svmult}

\usepackage{mathptmx}       % selects Times Roman as basic font
\usepackage{helvet}         % selects Helvetica as sans-serif font
\usepackage{courier}        % selects Courier as typewriter font
\usepackage{type1cm}        % activate if the above 3 fonts are
\usepackage{makeidx}         % allows index generation
\usepackage{graphicx}        % standard LaTeX graphics tool
\usepackage{multicol}        % used for the two-column index
\usepackage[bottom]{footmisc}% places footnotes at page bottom

\newtheorem {thm}{Theorem}
\newtheorem {cor}{Corollary}
\newtheorem {lem}{Lemma}

\newtheorem {defn}{Definition}

\makeindex             % used for the subject index
\begin{document}

\title*{A Multidimensional Hilbert-Type Integral Inequality Related to the Riemann Zeta Function}
\author{Michael Th. Rassias$^{1}$ and Bicheng Yang$^{2}$ %\and 1. Department
%of Mathematics, ETH-Zentrum, \and CH-8092 Zurich, Switzerland \and %
%michail.rassias@math.ethz.ch \and 2. Department of Mathematics, Guangdong
%University of \and Education, Guangzhou, Guangdong 510303, P. R. China \and %
%bcyang@gdei.edu.cn
}
\institute{Michael Th. Rassias$^{1}$ \at Department
of Mathematics, ETH-Zentrum, CH-8092 Zurich, Switzerland, \email{michail.rassias@math.ethz.ch}
\and Bicheng Yang$^{2}$ \at Department of Mathematics, Guangdong
University of, \and Education, Guangzhou, Guangdong 510303, P. R. China \email{bcyang@gdei.edu.cn}}
\maketitle

\abstract{\,\,In this chapter, using methods of weight functions and
techniques of real analysis, we provide a multidimensional
Hilbert-type integral inequality with a homogeneous kernel of
degree 0 as well as a best possible constant factor related to the
Riemann zeta function. Some equivalent representations and certain
reverses are obtained. Furthermore, we also consider operator
expressions with the norm and some particular results. }

%\abstract{Each chapter should be preceded by an abstract (10--15 lines long) that summarizes the content. The abstract will appear \textit{online} at \url{www.SpringerLink.com} and be available with unrestricted access. This allows unregistered users to read the abstract as a teaser for the complete chapter. As a general rule the abstracts will not appear in the printed version of your book unless it is the style of your particular book or that of the series to which your book belongs.\newline\indent
%Please use the 'starred' version of the new Springer \texttt{abstract} command for typesetting the text of the online abstracts (cf. source file of this chapter template \texttt{abstract}) and include them with the source files of your manuscript. Use the plain \texttt{abstract} command if the abstract is also to appear in the printed version of the book.}

\section{Introduction}
If $p>1,\frac{1}{p}+\frac{1}{q}=1,f(x),g(y)\geq 0,f\in L^{p}(\mathbf{R}%
_{+}),g\in L^{q}(\mathbf {R}_{+}),$ $$||f||_{p}=\left\{\int_{0}^{\infty }f^{p}(x)dx\right\}^{%
\frac{1}{p}}>0,$$
$||g||_{q}>0,$ then we have the following Hardy-Hilbert's
integral inequality (cf. \cite{HLP}):
\begin{equation}
\int_{0}^{\infty }\int_{0}^{\infty }\frac{f(x)g(y)}{x+y}dxdy<\frac{\pi }{%
\sin (\pi /p)}||f||_{p}||g||_{q},  \label{1}
\end{equation}%
where the constant factor $\frac{\pi }{\sin (\pi /p)}$ is the best possible.
If $a_{m},b_{n}\geq 0,a=\{a_{m}\}_{m=1}^{\infty }\in
l^{p},b=\{b_{n}\}_{n=1}^{\infty }\in l^{q},$ $||a||_{p}=\{\sum_{m=1}^{\infty
}a_{m}^{p}\}^{\frac{1}{p}}>0,||b||_{q}>0,$ then we still have the following
discrete variant of the above inequality with the same best constant $\frac{%
\pi }{\sin (\pi /p)},$ that is%
\begin{equation}
\sum_{m=1}^{\infty }\sum_{n=1}^{\infty }\frac{a_{m}b_{n}}{m+n}<\frac{\pi }{%
\sin (\pi /p)}||a||_{p}||b||_{q}.  \label{2}
\end{equation}%
Inequalities (\ref{1}) and (\ref{2}) are important in mathematical
analysis and its applications (cf. \cite{HLP}, \cite{MPF},
\cite{BAC}, \cite{Y1}, \cite{Y2}, \cite{Y4}, \cite{YB7}).

In 1998, by introducing an independent parameter $\lambda \in
(0,1]$, Yang \cite{Y3} presented an extension of (\ref{1}) for
$p=q=2$. In 2009 and 2011, Yang \cite{Y1}, \cite{Y2} provided some
extensions of (\ref{1}) and (\ref{2}) as follows: If $\lambda
_{1},\lambda _{2},\lambda \in \mathbf{R},\lambda _{1} + \lambda
_{2}=\lambda ,k_{\lambda }(x,y)$ is a non-negative homogeneous
function of degree $-\lambda ,$ with%
 $$
k(\lambda _{1})=\int_{0}^{\infty }k_{\lambda }(t,1)t^{\lambda
_{1}-1}dt\in \mathbf {R}_{+},
 $$
$\phi (x)=x^{p(1-\lambda _{1})-1},\psi (y)=y^{q(1-\lambda
_{2})-1},f(x),g(y)\geq 0,$
 $$
f\in L_{p,\phi }(\mathbf {R}_{+})=\left\{ f:||f||_{p,\phi
}:=\left\{\int_{0}^{\infty }\phi
(x)|f(x)|^{p}dx\right\}^{\frac{1}{p}}<\infty \right\} ,
 $$
$g\in L_{q,\psi }(\mathbf{R}_{+}),||f||_{p,\phi },||g||_{q,\psi
}>0,$ then we have
\begin{equation}
\int_{0}^{\infty }\int_{0}^{\infty }k_{\lambda }(x,y)f(x)g(y)dxdy<k(\lambda
_{1})||f||_{p,\phi }||g||_{q,\psi },  \label{3}
\end{equation}%
where the constant factor $k(\lambda _{1})$ is the best possible. Moreover,
if $k_{\lambda }(x,y)$ is finite and $k_{\lambda }(x,y)x^{\lambda
_{1}-1}(k_{\lambda }(x,y)y^{\lambda _{2}-1})$ is decreasing with respect to $%
x>0\: (y>0),$ then for $a_{m,}b_{n}\geq 0,$%
 $$
a\in l_{p,\phi }=\left\{ a:||a||_{p,\phi
}:=\left\{\sum_{n=1}^{\infty }\phi
(n)|a_{n}|^{p}\right\}^{\frac{1}{p}}<\infty \right\} ,
 $$
$b=\{b_{n}\}_{n=1}^{\infty }\in l_{q,\psi },$ $||a||_{p,\phi },||b||_{q,\psi
}>0,$ we have
\begin{equation}
\sum_{m=1}^{\infty }\sum_{n=1}^{\infty }k_{\lambda
}(m,n)a_{m}b_{n}<k(\lambda _{1})||a||_{p,\phi }||b||_{q,\psi },  \label{4}
\end{equation}%
where, the constant factor $k(\lambda _{1})$ is still the best possible.

Clearly, for $\lambda =1,k_{1}(x,y)=\frac{1}{x+y},$ $\lambda _{1}=\frac{1}{q}%
,\lambda _{2}=\frac{1}{p},$ (\ref{3}) reduces to (\ref{1}), while (\ref{4})
reduces to (\ref{2}). Some further results including a few multidimensional
Hilbert-type integral inequalities are provided in \cite{YMP}-\cite{LH}.

In this chapter, using methods of weight functions and techniques
of real analysis, we present a new multidimensional Hilbert-type integral
inequality with a homogeneous kernel of degree 0 as well as a best possible
constant factor related to the Riemann zeta function and the Gamma function, which is an extension of
the double case as follows:%
\begin{equation}
\int_{0}^{\infty }\int_{0}^{\infty }\left( \coth (\frac{x}{y})-1\right)
f(x)g(y)dxdy<\frac{\Gamma (\sigma )}{2^{\sigma -1}}\zeta (\sigma
)||f||_{p,\varphi }||g||_{q,\psi },  \label{5}
\end{equation}%
where, $\zeta (\cdot )$ is the Riemann zeta function and $\Gamma(\cdot)$ is the Gamma function (cf. \cite{R3}, \cite{R2}). Some equivalent forms and reverses are obtained. Furthermore, we also consider the
operator expressions with the norm and certain particular results. For a number of fundamental properties of the Riemann zeta function and the Gamma function, especially in Analytic Number Theory and related subjects, the reader is referred to \cite{R1}-\cite{R8}, \cite{RR}.

\section{Some Lemmas}
If $m,n\in \mathbf{N(N}$ is the set of positive integers), $\alpha
,\beta >0,$ we define%
 $$
||x||_{\alpha }:=\left( \sum_{k=1}^{m}|x_{k}|^{\alpha }\right) ^{\frac{1%
}{\alpha }}(x=(x_{1},\cdots ,x_{m})\in \mathbf{R}^{m}),
 $$
 $$
||y||_{\beta }:=\left( \sum_{k=1}^{n}|y_{k}|^{\beta }\right) ^{\frac{1}{%
\beta }}(y=(y_{1},\cdots ,y_{n})\in \mathbf{R}^{n}).
 $$
\begin{lem}
 If $s\in \mathbf{N},\gamma ,M>0,\Psi (u)$ is a
non-negative measurable function defined in $(0,1],$ and
 $$
D_{M}^{s}:=\left\{ x\in \mathbf{R}_{+}^{s}:0<u=\sum_{i=1}^{s}\left( \frac{x_{i}}{%
M}\right) ^{\gamma }\leq 1\right\} ,
 $$
then we have (cf. \cite{YB7})%
\begin{equation}
\int \cdots \int_{D_{M}^{s}}\Psi \left( \sum_{i=1}^{s}\left( \frac{x_{i}}{M}%
\right) ^{\gamma }\right) dx_{1}\cdots dx_{s}=\frac{M^{s}\Gamma ^{s}(\frac{1}{\gamma })}{\gamma ^{s}\Gamma (\frac{s}{%
\gamma })}\int_{0}^{1}\Psi (u)u^{\frac{s}{\gamma }-1}du.  \label{6}
\end{equation}
\end{lem}
\begin{lem}[See \cite{YM}]
  If $s\in \mathbf{N,}\gamma >0,$ and $\varepsilon \geq0,$ then
\begin{equation}
\int \cdots \int_{\{x\in \mathbf{R}_{+}^{s}:||x||_{\gamma }\geq
1\}}||x||_{\gamma }^{-s-\varepsilon }dx_{1}\cdots dx_{s}=
\left\{\begin{array}{ll}
\frac{\Gamma ^{s}(%
\frac{1}{\gamma })}{\varepsilon \gamma ^{s-1}\Gamma
(\frac{s}{\gamma })}& \varepsilon>0\\
\infty& \varepsilon=0 \end{array}\right.. \label{9}
\end{equation}
\end{lem}
\begin{defn}
 For $x=(x_{1},\cdots ,x_{m})\in \mathbf{R}%
_{+}^{m},y=(y_{1},\cdots ,y_{n})\in \mathbf{R}_{+}^{n}$, $\sigma
>1,$ we define two weight functions $\omega (\sigma ,y)$ and $\varpi (\sigma
,x)$, as follows%
\begin{equation}
\omega (\sigma ,y):=||y||_{\beta }^{-\sigma }\int_{\mathbf{R}%
_{+}^{m}}\left( \coth \frac{||x||_{\alpha }}{||y||_{\beta
}}-1\right) \frac{dx}{||x||_{\alpha }^{m-\sigma }}, \label{10}
\end{equation}%
\begin{equation}
\varpi (\sigma ,x):=||x||_{\alpha }^{\sigma }\int_{\mathbf{R}%
_{+}^{n}}\left( \coth \frac{||x||_{\alpha }}{||y||_{\beta
}}-1\right) \frac{dy}{||y||_{\beta }^{n+\sigma }},  \label{11}
\end{equation}%
where $\coth u=\frac{e^{u}+e^{-u}}{e^{u}-e^{-u}}$ is the hyperbolic
cotangent function (cf. \cite{ZHC}).
\end{defn}

By (\ref{6}), setting $v=\frac{Mu^{\frac{1}{\alpha
}}}{||y||_{\beta }},$ we
find%
\begin{eqnarray*}
\omega (\sigma ,y)&=&||y||_{\beta }^{-\sigma }\lim_{M\rightarrow
\infty }\int_{D_{M}^m}\left( \coth \frac{||x||_{\alpha
}}{||y||_{\beta
}}-1\right) \frac{dx}{||x||_{\alpha }^{m-\sigma }}\\
&=&||y||_{\beta }^{-\sigma }\lim_{M\rightarrow \infty
}\int_{D_{M}^m}%
\frac{\coth \frac{M}{||y||_{\beta }}[\sum_{i=1}^{m}(\frac{x_{i}}{M}%
)^{\alpha }]^{\frac{1}{\alpha }}-1}{M^{m-\sigma
}[\sum_{i=1}^{m}\left( \frac{x_{i}}{M}\right) ^{\alpha }]^{\frac{%
m-\sigma }{\alpha }}}dx \\
&=&||y||_{\beta }^{-\sigma }\lim_{M\rightarrow \infty
}\frac{M^{m}\Gamma
^{m}(\frac{1}{\alpha })}{\alpha ^{m}\Gamma (\frac{m}{\alpha })}%
\int_{0}^{1}\frac{\coth (\frac{M}{||y||_{\beta }})u^{\frac{1}{\alpha }}-1}{%
M^{m-\sigma }u^{\frac{m-\sigma }{\alpha }}}u^{\frac{m}{\alpha }%
-1}du \\
&=&||y||_{\beta }^{-\sigma }\lim_{M\rightarrow \infty }\frac{M^{\sigma
}\Gamma ^{m}(\frac{1}{\alpha })}{\alpha ^{m}\Gamma (\frac{m}{%
\alpha })}\int_{0}^{1}\left( \coth (\frac{M}{||y||_{\beta }})u^{\frac{1}{%
\alpha }}-1\right) u^{\frac{\sigma }{\alpha }-1}du \\
&=&\frac{\Gamma ^{m}(\frac{1}{\alpha })}{\alpha ^{m-1}\Gamma (\frac{%
m}{\alpha })}\int_{0}^{\infty }(\coth v-1)v^{\sigma -1}dv,
\end{eqnarray*}
and in view of the Lebesgue term by term theorem (cf. \cite{K2}), it follows
\begin{eqnarray}
\int_{0}^{\infty }(\coth v-1)v^{\sigma -1}dv 
&=\int_{0}^{\infty
}\left(\frac{e^{v}+e^{-v}}{e^{v}-e^{-v}}-1\right)v^{\sigma
-1}dv=\int_{0}^{\infty }\frac{2e^{-2v}v^{\sigma -1}}{1-e^{-2v}}dv \nonumber \\
&=2\int_{0}^{\infty }\sum_{k=1}^{\infty }e^{-2kv}v^{\sigma
-1}dv=2\sum_{k=1}^{\infty }\int_{0}^{\infty }e^{-2kv}v^{\sigma
-1}dv\nonumber \\
&=2\sum_{k=1}^{\infty }\frac{1}{(2k)^{\sigma }}\Gamma (\sigma
)=\frac{\Gamma (\sigma )}{2^{\sigma -1}}\zeta (\sigma ),
\label{12}
\end{eqnarray}
where $\zeta (\sigma )=\sum_{k=1}^{\infty }\frac{1}{k^{\sigma }},\ \sigma >1.$%
\begin{lem}
For $\sigma ,\widetilde{\sigma }>1,$ we have%
\begin{eqnarray}
\omega (\sigma ,y) &=&K_{2}(\sigma ):=\frac{\Gamma ^{m}(\frac{1}{\alpha }%
)}{\alpha ^{m-1}\Gamma (\frac{m}{\alpha })}\frac{\Gamma (\sigma )}{%
2^{\sigma -1}}\zeta (\sigma )(y\in \mathbf{R}_{+}^{n}),  \label{13} \\
\varpi (\sigma ,x) &=&K_{1}(\sigma ):=\frac{\Gamma ^{n}(\frac{1}{\beta })%
}{\alpha ^{n-1}\Gamma (\frac{n}{\beta })}\frac{\Gamma (\sigma )}{%
2^{\sigma -1}}\zeta (\sigma )(x\in \mathbf{R}_{+}^{m}), \label{14}
\end{eqnarray}%
\begin{eqnarray*}
w(\widetilde{\sigma },y) &:=&||y||_{\beta }^{-\widetilde{\sigma }%
}\int_{\{x\in \mathbf{R}_{+}^{m}:||x||_{\alpha }\geq 1\}}\left(
\coth \frac{||x||_{\alpha }}{||y||_{\beta }}-1\right)
\frac{dx}{||x||_{\alpha
}^{m-\widetilde{\sigma }}} \\
&=& K_{2}(\widetilde{\sigma })\left[ 1-\theta _{\widetilde{\sigma }%
}(||y||_{\beta })\right],
\end{eqnarray*}%
and
\begin{eqnarray}
\theta _{\widetilde{\sigma }}(||y||_{\beta })
&:=&\frac{2^{\widetilde{\sigma
}-1}}{\Gamma (\widetilde{\sigma })\zeta (\widetilde{\sigma })}%
\int_{0}^{||y||_{\beta }^{-1}}(\coth v-1)v^{\widetilde{\sigma
}-1}dv\nonumber\\
&=&O(||y||_{\beta }^{-\widetilde{\eta }})(\widetilde{\eta }>0;y\in
\mathbf{R}_{+}^{n}).  \label{15}
\end{eqnarray}
\end{lem}
\begin{proof}
 By (\ref{12}), we obtain (\ref{13}). Similarly, we get (%
\ref{14}). By (\ref{6}) for $\Psi(u)=0(u\in(0,1/M^\gamma))$, we find%
\begin{eqnarray*}
w(\widetilde{\sigma },y) &=&\frac{\Gamma ^{m}(\frac{1}{\alpha
})}{\alpha ^{m-1}\Gamma (\frac{m}{\alpha })}\int_{||y||_{\beta
}^{-1}}^{\infty
}(\coth v-1)v^{\widetilde{\sigma }-1}dv \\
&=&\frac{\Gamma ^{m}(\frac{1}{\alpha })}{\alpha ^{m-1}\Gamma (\frac{%
m}{\alpha })} \\
&&\times \left[ \int_{0}^{\infty }(\coth v-1)v^{\widetilde{\sigma }%
-1}dv-\int_{0}^{||y||_{\beta }^{-1}}(\coth v-1)v^{\widetilde{\sigma }-1}dv%
\right] \\
&=&\frac{\Gamma ^{m}(\frac{1}{\alpha })}{\alpha ^{m-1}\Gamma (\frac{%
m}{\alpha })}\frac{\Gamma (\widetilde{\sigma })}{2^{\widetilde{\sigma }%
-1}}\zeta (\widetilde{\sigma })\left[ 1-\theta _{\widetilde{\sigma }%
}(||y||_{\beta })\right] .
\end{eqnarray*}
Considering a constant $\gamma \in (1,\widetilde{\sigma }),$ we obtain%
\begin{eqnarray*}
\lim_{v\rightarrow 0^{+}}(\coth v-1)v^{\gamma } &=&\lim_{v\rightarrow 0^{+}}%
\frac{2v^{\gamma }}{e^{2v}-1}=\lim_{v\rightarrow 0^{+}}\frac{2\gamma
v^{\gamma -1}}{2e^{2v}}=0, \\
\lim_{v\rightarrow \infty }(\coth v-1)v^{\gamma } &=&0.
\end{eqnarray*}%
There exists a constant $L>0,$ such that $(\coth v-1)\leq Lv^{-\gamma }.$\\

Setting $\widetilde{\eta }:=\widetilde{\sigma }-\gamma (>0),$ it follows%
 $$
0\leq \theta _{\widetilde{\sigma }}(||y||_{\beta })\leq \frac{2^{\widetilde{%
\sigma }-1}L}{\Gamma (\widetilde{\sigma })\zeta (\widetilde{\sigma })}%
\int_{0}^{||y||_{\beta }^{-1}}v^{\widetilde{\eta }-1}dv=\frac{2^{\widetilde{%
\sigma }-1}L}{\Gamma (\widetilde{\sigma })\zeta (\widetilde{\sigma })%
\widetilde{\eta }}\frac{1}{||y||_{\beta }^{\widetilde{\eta }}},
 $$
and then $$\theta _{\widetilde{\sigma }}(||y||_{\beta })=O(
||y||_{\beta }^{-\widetilde{\eta }}) (y\in \mathbf{R}_{+}^{n}).$$
This completes the proof of the lemma.
\end{proof}
\begin{lem}
By the assumptions of Definition 1, if $p\in \mathbf{R}%
\backslash \{0,1\},\frac{1}{p}+\frac{1}{q}=1,$
$f(x)=f(x_{1},\cdots ,x_{m})\geq 0,$ $g(y)=g(y_{1},\cdots
,y_{n})\geq 0,$ then

(i) for $%
p>1,$ we have the following inequality:%
\begin{eqnarray}
J_{1}&:=&\left\{ \int_{\mathbf{R}_{+}^{n}}\frac{||y||_{\beta
}^{-p\sigma -n}}{[\omega (\sigma
,y)]^{p-1}}\left[\int_{\mathbf{R}_{+}^{m}} \left( \coth
\frac{||x||_{\alpha }}{||y||_{\beta }}-1\right) f(x)dx\right]
^{p}dy\right\} ^{\frac{1}{p}}\nonumber\\
&\leq& \left\{ \int_{\mathbf{R}_{+}^{m}}\varpi (\sigma
,x)||x||_{\alpha }^{p(m-\sigma )-m}f^{p}(x)dx\right\}
^{\frac{1}{p}},  \label{16}
\end{eqnarray}

(ii) for $0<p<1$ or $p<0,$ we obtain the reverses of (\ref{16}).
\end{lem}
\begin{proof}
  (i) For $p>1,$ by H\"{o}lder's inequality with weight (cf.
\cite{K1}), it follows%
\begin{eqnarray}
&\int_{\mathbf{R}_{+}^{m}}\left( \coth \frac{||x||_{\alpha }}{%
||y||_{\beta }}-1\right) f(x)dx \nonumber\\
&=\int_{\mathbf{R}_{+}^{m}}\left( \coth \frac{||x||_{\alpha }}{%
||y||_{\beta }}-1\right) \left[ \frac{||x||_{\alpha }^{(m-\sigma )/q}}{%
||y||_{\beta }^{(n+\sigma )/p}}f(x)\right] \left[
\frac{||y||_{\beta }^{(n+\sigma )/p}}{||x||_{\alpha }^{(m-\sigma
)/q}}\right] dx \nonumber\\
&\leq \left\{ \int_{\mathbf{R}_{+}^{m}}\left( \coth
\frac{||x||_{\alpha
}}{||y||_{\beta }}-1\right) \frac{||x||_{\alpha }^{(m-\sigma )(p-1)}}{%
||y||_{\beta }^{n+\sigma }}f^{p}(x)dx\right\} ^{\frac{1}{p}} \nonumber\\
&\quad\times \left\{ \int_{\mathbf{R}_{+}^{m}}\left( \coth \frac{%
||x||_{\alpha }}{||y||_{\beta }}-1\right) \frac{||y||_{\beta
}^{(n+\sigma )(q-1)}}{||x||_{\alpha }^{m-\sigma }}dx\right\} ^{\frac{%
1}{q}} \nonumber\\
&=[\omega (\sigma ,y)]^{\frac{1}{q}}||y||_{\beta
}^{\frac{n}{p}+\sigma }\nonumber\\
&\quad \times \left\{ \int_{\mathbf{R}_{+}^{m}}\left( \coth
\frac{||x||_{\alpha
}}{||y||_{\beta }}-1\right) \frac{||x||_{\alpha }^{(m-\sigma )(p-1)}}{%
||y||_{\beta }^{n+\sigma }}f^{p}(x)dx\right\} ^{\frac{1}{p}}.
\label{17}
\end{eqnarray}%
Then by Fubini's theorem (cf. \cite{K2}), we have%
\begin{eqnarray}
J_{1} &\leq &\left\{ \int_{\mathbf{R}_{+}^{n}}\left[ \int_{\mathbf{R}%
_{+}^{m}}\left( \coth \frac{||x||_{\alpha }}{||y||_{\beta
}}-1\right)
\frac{||x||_{\alpha }^{(m-\sigma )(p-1)}}{||y||_{\beta }^{n+\sigma }}%
f^{p}(x)dx\right] dy\right\} ^{\frac{1}{p}}  \nonumber\\
&=&\left\{ \int_{\mathbf{R}_{+}^{m}}\left[ \int_{\mathbf{R}%
_{+}^{n}}\left( \coth \frac{||x||_{\alpha }}{||y||_{\beta
}}-1\right)
\frac{||x||_{\alpha }^{(m-\sigma )(p-1)}}{||y||_{\beta }^{n+\sigma }}%
dy\right] f^{p}(x)dx\right\} ^{\frac{1}{p}}\nonumber\\
&=&\left\{ \int_{\mathbf{R}_{+}^{m}}\varpi (\sigma
,x)||x||_{\alpha }^{p(m-\sigma )-m}f^{p}(x)dx\right\}
^{\frac{1}{p}}. \label{18}
\end{eqnarray}
Hence, (\ref{16}) follows.

(ii) For $0<p<1$ or $p<0,$ by the reverse H\"{o}lder inequality with
weight (cf. \cite{K1}), we obtain the reverse of (\ref{17}). Then by Fubini's
theorem, we can still obtain the reverse of (\ref{16}) and thus the lemma is proved.
\end{proof}
\begin{lem}
By the assumptions of Lemma 4,

(i) for $p>1,$ we have the following inequality equivalent to (\ref{16}):%
\begin{eqnarray}
I &:=&\int_{\mathbf{R}_{+}^{n}}\int_{\mathbf{R}_{+}^{m}}\left(
\coth
\frac{||x||_{\alpha }}{||y||_{\beta }}-1\right) f(x)g(y)dxdy\nonumber\\
&\leq &\left\{ \int_{\mathbf{R}_{+}^{m}}\varpi (\sigma
,x)||x||_{\alpha }^{p(m-\sigma )-m}f^{p}(x)dx\right\}
^{\frac{1}{p}}\nonumber\\
&&\times \left\{ \int_{\mathbf{R}_{+}^{n}}\omega (\sigma
,y)||y||_{\beta }^{q(n+\sigma )-n}g^{q}(y)dy\right\}
^{\frac{1}{q}}, \label{19}
\end{eqnarray}%

(ii) for $0<p<1$ or $p<0,$ we have the reverse of (\ref{19}) equivalent to
the reverses of (\ref{16}).
\end{lem}
\begin{proof}
  (i) For $p>1,$ by H\"{o}lder's inequality (cf. \cite{K1}), it
follows that%
\begin{eqnarray}
I &=&\int_{\mathbf{R}_{+}^{n}}\frac{||y||_{\beta }^{\frac{n}{q}%
-(n+\sigma )}}{[\omega (\sigma ,y)]^{\frac{1}{q}}}\left[ \int_{\mathbf{R}%
_{+}^{m}}\left( \coth \frac{||x||_{\alpha }}{||y||_{\beta
}}-1\right)
f(x)dx\right]\nonumber\\
&&\times \left[ \lbrack \omega (\sigma
,y)]^{\frac{1}{q}}||y||_{\beta }^{(n+\sigma
)-\frac{n}{q}}g(y)\right]dy\nonumber\\
&\leq& J_{1}\left\{ \int_{\mathbf{R}_{+}^{n}}\omega (\sigma
,y)||y||_{\beta }^{q(n+\sigma )-n}g^{q}(y)dy\right\}
^{\frac{1}{q}}. \label{20}
\end{eqnarray}%
Then by (\ref{16}), we obtain (\ref{19}).

On the other hand, assuming that (\ref{19}) is valid, we set%
 $$
g(y):=\frac{||y||_{\beta }^{-p\sigma -n}}{[\omega (\sigma ,y)]^{p-1}}%
\left( \int_{\mathbf{R}_{+}^{m}}\left( \coth \frac{||x||_{\alpha }}{%
||y||_{\beta }}-1\right) f(x)dx\right) ^{p-1},y\in
\mathbf{R}_{+}^{n}.
 $$
Then it follows that
 $$
J_{1}^{p}=\int_{\mathbf{R}_{+}^{n}}\omega (\sigma ,y)||y||_{\beta
}^{q(n+\sigma )-n}g^{q}(y)dy.
 $$
If $J_{1}=0,$ then (\ref{16}) is trivially valid; if $J_{1}=\infty ,$ then
by (\ref{18}), relation (\ref{16}) reduces to the form of an equality($=\infty )$. Suppose
that $0<J_{1}<\infty .$ By (\ref{19}), we have%
\begin{eqnarray*}
0 &<&\int_{\mathbf{R}_{+}^{n}}\omega (\sigma ,y)||y||_{\beta
}^{q(n+\sigma )-n}g^{q}(y)dy=J_{1}^{p}=I \\
&\leq &\left\{ \int_{\mathbf{R}_{+}^{m}}\varpi (\sigma
,x)||x||_{\alpha
}^{p(m-\sigma )-m}f^{p}(x)dx\right\} ^{\frac{1}{p}} \\
&&\times \left\{ \int_{\mathbf{R}_{+}^{n}}\omega (\sigma
,y)||y||_{\beta }^{q(n+\sigma )-n}g^{q}(y)dy\right\}
^{\frac{1}{q}}<\infty .
\end{eqnarray*}%
Therefore,%
\begin{eqnarray*}
J_{1} &=&\left\{ \int_{\mathbf{R}_{+}^{n}}\omega (\sigma
,y)||y||_{\beta
}^{q(n+\sigma )-n}g^{q}(y)dy\right\} ^{\frac{1}{p}} \\
&\leq &\left\{ \int_{\mathbf{R}_{+}^{m}}\varpi (\sigma
,x)||x||_{\alpha }^{p(m-\sigma )-m}f^{p}(x)dx\right\}
^{\frac{1}{p}},
\end{eqnarray*}%
and then (\ref{16}) follows. Hence, (\ref{16}) and (\ref{19}) are equivalent.

(ii) For $0<p<1$ or $p<0,$ similarly, we obtain the reverse of (\ref{19}%
) which is equivalent to the reverse of (\ref{16}) and thus the lemma is proved.
\end{proof}
\section{Main Results and Operator Expressions}

Let
$$\Phi (x):=||x||_{\alpha }^{p(m-\sigma )-m},\Psi (y):=||y||_{\beta
}^{q(n+\sigma )-n}(x\in \mathbf{R}_{+}^{m},y\in \mathbf{R}%
_{+}^{n}),
 $$
by Lemma 3, Lemma 4 and Lemma 5, we obtain

\begin{thm}
Suppose that $\alpha ,\beta >0,\ \sigma >1,p\in \mathbf{R%
}\backslash \{0,1\},\frac{1}{p}+\frac{1}{q}=1,$
$f(x)\\=f(x_{1},\cdots ,x_{m})\geq 0,\ g(y)=g(y_{1},\cdots
,y_{n})\geq 0,$
$$0<||f||_{p,\Phi }=\left\{\int_{\mathbf{R}_{+}^{m}}\Phi (x)f^{p}(x)dx\right\}^{%
\frac{1}{p}}<\infty ,$$
and
$$0<||g||_{q,\Psi }=\left\{\int_{\mathbf{R}_{+}^{n}}\Psi (y)g^{q}(y)dy\right\}^{%
\frac{1}{q}}<\infty.$$

(i) If $p>1,$ then we have the following equivalent inequalities with the
best possible constant factor $K(\sigma ),$ that is%
\begin{equation}
I=\int_{\mathbf{R}_{+}^{n}}\int_{\mathbf{R}_{+}^{m}}\left( \coth
\frac{||x||_{\alpha }}{||y||_{\beta }}-1\right)
f(x)g(y)dxdy<K(\sigma )||f||_{p,\Phi }||g||_{q,\Psi },  \label{21}
\end{equation}%
and
\begin{eqnarray}
J&:=&\left\{ \int_{\mathbf{R}_{+}^{n}}||y||_{\beta }^{-p\sigma
-n}\left( \int_{\mathbf{R}_{+}^{m}}\left( \coth \frac{||x||_{\alpha }%
}{||y||_{\beta }}-1\right) f(x)dx\right) ^{p}dy\right\} ^{\frac{1}{p}}\nonumber\\
&<&K(\sigma )||f||_{p,\Phi },  \label{22}
\end{eqnarray}%
where%
\begin{equation}
K(\sigma )=\left[ \frac{\Gamma ^{n}(\frac{1}{\beta })}{\beta ^{n-1}\Gamma (%
\frac{n}{\beta })}\right] ^{\frac{1}{p}}\left[ \frac{\Gamma ^{m}(%
\frac{1}{\alpha })}{\alpha ^{m-1}\Gamma (\frac{m}{\alpha
})}\right] ^{\frac{1}{q}}\frac{\Gamma (\sigma )}{2^{\sigma
-1}}\zeta (\sigma ). \label{23}
\end{equation}

(ii) If $0<p<1$ or $p<0,$ then we still have the equivalent reverses of (%
\ref{21}) and (\ref{22}) with the same best constant factor $K(\sigma ).$
\end{thm}
\begin{proof}
  (i) For $p>1$, by the conditions, we can prove that (\ref{17})
becomes a strict inequality. Otherwise if (\ref{17}) takes the
form of equality, then there exist constants $A$ and $B$, which
are not all zero, such that for a.e. $y\in \mathbf{R}_{+}^{n},$
\begin{equation}
A\frac{||x||_{\alpha }^{(m-\sigma )(p-1)}}{||y||_{\beta }^{n+\sigma }%
}f^{p}(x)=B\frac{||y||_{\beta }^{(n+\sigma )(q-1)}}{||x||_{\alpha
}^{m-\sigma }}\,\, {\rm a.e.\,\, in\,\, }x\in \mathbf{R}_{+}^{m}.
\label{24}
\end{equation}%
If $A=0,$ then it follows that $B=0,$ which is impossible; if $A\neq 0,$ then (\ref{24}) reduces to
 $$
||x||_{\alpha }^{p(m-\sigma )-m}f^{p}(x)=\frac{B||y||_{\beta
}^{q(n+\sigma )}}{A||x||_{\alpha }^{m}}\,\,{\rm a.e.\,\, in \,\,}x\in \mathbf{R}%
_{+}^{m},
 $$
which contradicts $0<||f||_{p,\Phi }<\infty .$ In fact by (%
\ref{9}), it follows $$\int_{\mathbf{R}_{+}^{m}}||x||_{\alpha
}^{-m}dx=\infty .$$ Hence (\ref{16}) still assumes the form of
strict inequality. By Lemma 3 and Lemma 4, we deduce (\ref{22}).

Similarly to (\ref{20}), we still have%
\begin{equation}
I\leq J\left\{ \int_{\mathbf{R}_{+}^{n}}||y||_{\beta }^{q(n+\sigma
)-n}g^{q}(y)dy\right\} ^{\frac{1}{q}}. \label{25}
\end{equation}%
Then by (\ref{25}) and (\ref{22}), we obtain (\ref{21}). It is evident by
Lemma 5 and the assumptions, that the relations (\ref{21}) and (\ref{20}) are also
equivalent.

For $0<\varepsilon <p(\sigma -1),$ we define $\widetilde{f}(x),\widetilde{g}(y)$
as follows%
 $$
\widetilde{f}(x):=\left\{
\begin{array}{ll}
0,&0<||x||_{\alpha }<1, \\
||x||_{\alpha }^{\sigma -\frac{\varepsilon }{p}-m},&||x||_{\alpha }\geq 1,%
\end{array}%
\right.
 $$
 $$
\widetilde{g}(y):=\left\{
\begin{array}{ll}
0,&0<||y||_{\beta }<1, \\
||y||_{\beta }^{-\sigma -\frac{\varepsilon }{q}-n},&||y||_{\beta
}\geq
1.%
\end{array}%
\right.
 $$
Then for $\widetilde{\sigma }=\sigma -\frac{\varepsilon }{p},$ by (\ref{9}),
we derive
\begin{eqnarray*}
0 &\leq &\int_{\{y\in \mathbf{R}_{+}^{n}:||y||_{\beta }\geq
1\}}||y||_{\beta }^{-n-\varepsilon }O(||y||_{\beta }^{%
-\widetilde{\eta }})dy \\
&\leq &L^*\int_{\{y\in \mathbf{R}_{+}^{n}:||y||_{\beta }\geq
1\}}||y||_{\beta }^{-n-(\varepsilon +\widetilde{\eta })}dy \\
&=&L^*\frac{\Gamma ^{n}(\frac{1}{\beta })}{(\varepsilon +\widetilde{\eta }%
)\beta ^{n-1}\Gamma (\frac{n}{\beta })}<\infty ,
\end{eqnarray*}%
and in view of (\ref{9}) and (\ref{15}), it follows that%
\begin{eqnarray*}
&||\widetilde{f}||_{p,\Phi }||\widetilde{g}||_{q,\Psi } \\
&=\left\{ \int_{\{x\in \mathbf{R}_{+}^{m}:||x||_{\alpha }\geq
1\}}||x||_{\alpha }^{-m-\varepsilon }dx\right\}
^{\frac{1}{p}}\left\{ \int_{\{y\in \mathbf{R}_{+}^{n}:||y||_{\beta
}\geq 1\}}||y||_{\beta
}^{-n-\varepsilon }dy\right\} ^{\frac{1}{q}} \\
&=\frac{1}{\varepsilon }\left\{ \frac{\Gamma ^{m}(\frac{1}{\alpha })}{%
\alpha ^{m-1}\Gamma (\frac{m}{\alpha })}\right\} ^{\frac{1}{p}%
}\left\{ \frac{\Gamma ^{n}(\frac{1}{\beta })}{\beta ^{n-1}\Gamma (%
\frac{n}{\beta })}\right\} ^{\frac{1}{q}},
\end{eqnarray*}%
and
\begin{eqnarray*}
\widetilde{I} &:&=\int_{\mathbf{R}_{+}^{n}}\int_{\mathbf{R}%
_{+}^{m}}\left( \coth \frac{||x||_{\alpha }}{||y||_{\beta
}}-1\right)
\widetilde{f}(x)\widetilde{g}(y)dxdy \\
&=&\int_{\{y\in \mathbf{R}_{+}^{n}:||y||_{\beta }\geq
1\}}||y||_{\beta
}^{-n-\varepsilon }w(\widetilde{\sigma },y)dy \\
&=&K_{2}(\widetilde{\sigma })\int_{\{y\in \mathbf{R}_{+}^{n}:||y||_{%
\beta }\geq 1\}}||y||_{\beta }^{-n-\varepsilon }\left( 1-O(%
||y||_{\beta }^{-\widetilde{\eta }})\right) dy \\
&=&\frac{1}{\varepsilon }K_{2}(\widetilde{\sigma })\left[ \frac{\Gamma
^{n}(\frac{1}{\beta })}{\beta ^{n-1}\Gamma (\frac{n}{\beta })}%
-\varepsilon O_{\widetilde{\sigma }}(1)\right] .
\end{eqnarray*}

If there exists a constant $K\leq K(\sigma ),$ such that (\ref{21}) is valid
when replacing $K(\sigma )$ by $K,$ then we obtain%
\begin{eqnarray*}
&\frac{\Gamma ^{m}(\frac{1}{\alpha })}{\alpha ^{m-1}\Gamma (\frac{%
m}{\alpha })}\frac{\Gamma (\widetilde{\sigma })}{2^{\widetilde{\sigma }%
-1}}\zeta (\widetilde{\sigma })\left[ \frac{\Gamma ^{n}(\frac{1}{\beta })%
}{\beta ^{n-1}\Gamma (\frac{n}{\beta })}-\varepsilon O_{\widetilde{%
\sigma }}(1)\right] \\
&\leq \varepsilon \widetilde{I}<\varepsilon K||\widetilde{f}||_{p,\Phi }||%
\widetilde{a}||_{q,\Psi } \\
&=K\left\{ \frac{\Gamma ^{m}(\frac{1}{\alpha })}{\alpha
^{m-1}\Gamma (\frac{m}{\alpha })}\right\} ^{\frac{1}{p}}\left\{
\frac{\Gamma ^{n}(\frac{1}{\beta })}{\beta ^{n-1}\Gamma (\frac{n%
}{\beta })}\right\} ^{\frac{1}{q}},
\end{eqnarray*}%
and thus $K(\sigma )\leq K(\varepsilon \rightarrow 0^{+}).$ Hence $%
K=K(\sigma )$ is the best possible constant factor of (\ref{21}).

By the equivalency, we can prove that the constant factor $K(\sigma )$ in (%
\ref{22}) is the best possible. Otherwise, by
(\ref{25}) we would reach a contradiction to the fact that the constant factor $K(\sigma )$ in (\ref{21}) is the
best possible.

(ii) For $0<p<1$ or $p<0,$ similarly, we can still obtain the equivalent
reverses of (\ref{21}) and (\ref{22}) with the best constant factor. This completes the proof of the
theorem.
\end{proof}

\begin{cor}
Let the assumptions of Theorem 1 be fulfilled, and additionally,
$0<||f||_{1}:=\int_{\mathbf{R}_{+}^{m}}f(x)dx<\infty,$ and $0<||g||_{1}:=\int_{\mathbf{R}_{+}^{n}}g
(y)dy<\infty.$ Then,

(i) if $p>1,$ then we have the following equivalent inequalities with
the best possible constant factor $K(\sigma ),$ that is%
\begin{equation}
\int_{\mathbf{R}_{+}^{n}}\int_{\mathbf{R}_{+}^{m}}\coth \frac{%
||x||_{\alpha }}{||y||_{\beta }}f(x)g(y)dxdy<||f||_{1}||g||_{1}+K(\sigma
)||f||_{p,\Phi }||g||_{q,\Psi },  \label{26}
\end{equation}%
\begin{equation}
\left\{ \int_{\mathbf{R}_{+}^{n}}||y||_{\beta }^{-p\sigma
-n}\left(
\int_{\mathbf{R}_{+}^{m}}\coth \frac{||x||_{\alpha }}{||y||_{\beta }}%
f(x)dx-||f||_{1}\right) ^{p}dy\right\} ^{\frac{1}{p}}<K(\sigma
)||f||_{p,\Phi };  \label{27}
\end{equation}
%\end{cor}

(ii) if $0<p<1$ or $p<0,$ then we still have the equivalent reverses of (%
\ref{26}) and (\ref{27}) with the same best constant factor $K(\sigma ).$
\end{cor}
For $m=n=\alpha=\beta=1$ in Theorem 1 and Corollary 1, we obtain

\begin{cor}
 Suppose that $\sigma >1,p\in \mathbf{R}\backslash \{0,1\},\frac{%
1}{p}+\frac{1}{q}=1,$
 $$
\varphi (x):=x^{p(1-\sigma )-1},\psi (y):=y^{q(1+\sigma )-1}(x,y>0),
 $$
$f(x)\geq 0,$ $g(y)\geq 0,$ as well as $0<||f||_{p,\varphi
}<\infty,$ and $0<||g||_{q,\psi }<\infty.$ Then,

 (i) for $p>1,$ we have (\ref{5}) and the following equivalent
inequality with the best possible constant factor $\frac{\Gamma (\sigma )}{%
2^{\sigma -1}}\zeta (\sigma ),$ that is
\begin{equation}
\left\{ \int_{0}^{\infty }y^{-p\sigma -1}\left[ \int_{0}^{\infty }\left(
\coth\frac{x}{y}-1\right) f(x)dx\right] ^{p}dy\right\} ^{\frac{1}{p}}<%
\frac{\Gamma (\sigma )}{2^{\sigma -1}}\zeta (\sigma )||f||_{p,\varphi };
\label{28}
\end{equation}
%\end{cor}

(ii) for $0<p<1$ or $p<0,$ we obtain the equivalent reverses of (\ref{5}) and
(\ref{28}) with the same best constant factor.\textbf{\ }

Moreover, if
 $
0<||f||:=\int_{0}^{\infty }f(x)dx<\infty ,$ and
$0<||g||:=\int_{0}^{\infty }g(y)dy<\infty,$ then

(i) for $p>1,$ we have the following equivalent inequalities with the
best possible constant factor $\frac{\Gamma (\sigma )}{2^{\sigma -1}}\zeta
(\sigma ),$ that is%
\begin{equation}
\int_{0}^{\infty }\int_{0}^{\infty }\coth\frac{x}{y}f(x)g(y)dxdy<||f||%
\: ||g||+\frac{\Gamma (\sigma )}{2^{\sigma -1}}\zeta (\sigma
)||f||_{p,\varphi }||g||_{q,\psi },  \label{29}
\end{equation}%
\begin{equation}
\left\{ \int_{0}^{\infty }y^{-p\sigma -1}\left[ \int_{0}^{\infty }\coth
\frac{x}{y}f(x)dx-||f||\right] ^{p}dy\right\}
^{\frac{1}{p}}<\frac{\Gamma (\sigma )}{2^{\sigma -1}}\zeta (\sigma
)||f||_{p,\varphi },  \label{30}
\end{equation}

(ii) for $0<p<1$ or $p<0,$ we obtain the equivalent reverses of (\ref{29})
and (\ref{30}) with the same best constant factor.\textbf{\ }
\end{cor}
By the assumptions of Theorem 1 for $p>1,$ in view of $J<K(\sigma
)||f||_{p,\Phi },$ we define:

\begin{defn}
A multidimensional Hilbert-type integral
operator%
\begin{equation}
T:{L}_{p,\Phi }(\mathbf{R}_{+}^{m})\rightarrow {L}_{p,\Psi
^{1-p}}(\mathbf{R}_{+}^{n})  \label{31}
\end{equation}%
is defined as follows:

For $f\in{L}_{p,\Phi }(\mathbf{R}_{+}^{m}),$ there
exists a unique representation $Tf\in{L}_{p,\Psi ^{1-p}}(\mathbf{R}%
_{+}^{n}),$ satisfying%
\begin{equation}
(Tf)(y):=\int_{\mathbf{R}_{+}^{m}}\left( \coth \frac{||x||_{\alpha }}{%
||y||_{\beta }}-1\right) f(x)dx(y\in \mathbf{R}_{+}^{n}).
\label{32}
\end{equation}%
For $g\in{L}_{q,\Psi }(\mathbf{R}_{+}^{n}),$ we define the
following formal inner product of $Tf$ and $g$ as follows:%
\begin{equation}
(Tf,g):=\int_{\mathbf{R}_{+}^{n}}\int_{\mathbf{R}_{+}^{m}}\left(
\coth \frac{||x||_{\alpha }}{||y||_{\beta }}-1\right)
f(x)g(y)dxdy. \label{33}
\end{equation}
\end{defn}

Then by Theorem 1 for $p>1,0<||f||_{p,\Phi },||g||_{q,\Psi }<\infty ,$ we
have the following equivalent inequalities:%
\begin{equation}
(Tf,g)<K(\sigma )||f||_{p,\Phi }||g||_{q,\Psi },  \label{34}
\end{equation}
and
\begin{equation}
||Tf||_{p,\Psi ^{1-p}}<K(\sigma )||f||_{p,\Phi }. \label{35}
\end{equation}
It follows that $T$ is bounded with
 $$
||T||:=\sup_{f(\neq \theta )\in{L}_{p,\Phi }(\mathbf{R}_{+}^{m})}%
\frac{||Tf||_{p,\Psi ^{1-p}}}{||f||_{p,\Phi }}\leq K(\sigma ).
 $$
Since the constant factor $K(\sigma )$ in (\ref{35}) is the best possible,
we obtain%
 \begin{equation}
||T||=K(\sigma )=\left[ \frac{\Gamma ^{n}(\frac{1}{\beta })}{\beta
^{n-1}\Gamma (\frac{n}{\beta })}\right] ^{\frac{1}{p}} \left[
\frac{\Gamma ^{m}(\frac{1}{\alpha })}{\alpha ^{m-1}\Gamma
(\frac{m}{\alpha })}\right] ^{\frac{1}{q}}\frac{\Gamma (\sigma
)}{2^{\sigma -1}}\zeta (\sigma ).  \label{36}
\end{equation}

%\noindent\textbf{Acknowledgments.} The authors wish to express
%their thanks to Professors Tserendorj Batbold, Mario Krnic and Jichang Kuang for their careful reading of the manuscript
% and for their valuable suggestions.\\
% \textit{M. Th. Rassias:} This work is supported by the Greek State Scholarship
%Foundation (IKY). \textit{B. Yang:} This work is supported by 2012
%Knowledge Construction Special Foundation Item of Guangdong
%Institution of Higher Learning College and University (No.
%2012KJCX0079).

\begin{acknowledgement}
The authors wish to express
their thanks to Professors Tserendorj Batbold, Mario Krnic and Jichang Kuang for their careful reading of the manuscript
 and for their valuable suggestions.\\
 \textit{M. Th. Rassias:} This work is supported by the Greek State Scholarship
Foundation (IKY). \textit{B. Yang:} This work is supported by 2012
Knowledge Construction Special Foundation Item of Guangdong
Institution of Higher Learning College and University (No.
2012KJCX0079).
\end{acknowledgement}

\end{document}